\documentclass[12pt, reqno]{amsart}
\usepackage{amsfonts}
\usepackage{graphicx}
\usepackage{amscd}
\usepackage{amsmath}
\usepackage{amssymb}
\usepackage[margin=1in]{geometry}
\usepackage{color, hyperref}

\makeatletter
\@namedef{subjclassname@2010}{%
  \textup{2010} Mathematics Subject Classification.}
\makeatother

\theoremstyle{plain}

\numberwithin{equation}{section}
\newtheorem{theorem}{Theorem}[section]
\newtheorem*{theorem*}{Theorem}
\newtheorem{corollary}[theorem]{Corollary}
\newtheorem{lemma}[theorem]{Lemma}

\newtheorem{fact}[theorem]{Quick Fact}

\newcommand{\calO}{\mathcal{O}}
\DeclareMathOperator{\hess}{Hess}
\DeclareMathOperator{\ric}{Ric}
\newcommand{\R}{\mathbb{R}}

\title[On q-solitons] {On non-compact gradient solitons}

\author[Antonio W. Cunha and Erin Griffin]{Antonio W. Cunha$^1$ and Erin Griffin$^2$}

\address{$^1$ Departamento de P\'os-Gradua\c c\~ao em Matem\'{a}tica, Universidade Federal do Piau\'{\i}, 64049-550 Teresina, Piau\'i, Brazil.}
\email{wilsoncunha@ufpi.edu.br}

\address{$^2$ Department of Mathematics, Seattle Pacific University, Seattle, Washington, USA.}
\email{griffine@spu.edu}



\thanks{The work of the first author was partially supported by CNPq, Brazil (Grant: 430998/2018-0),  FAPEPI (PPP-Edital 007/ 2018) and PROPESQI/PRPG/UFPI-EDITAL Nº09/2022. The authors are grateful to Eric Bahuaud and Will Wylie for their helpful conversations while writing this paper.}

\begin{document}

\begin{abstract}
In this paper we extend existing results for generalized solitons, called $q$-solitons, to the complete case by considering non-compact solitons. By placing regularity conditions on the vector field $X$ and curvature conditions on $M$, we are able to use the chosen properties of the tensor $q$ to see that such non-compact $q$-solitons are stationary and $q$-flat. 

We conclude by applying our results to the examples of ambient obstruction solitons, Cotton solitons, and Bach solitons to demonstrate the utility of these general theorems for various flows. 

\end{abstract}

\maketitle

\section{Introduction and Main Results} \label{intro section}

The {\em $q$-flow} is the geometric flow defined by
\begin{equation} \label{q-flow}
\frac{\partial}{\partial t}g(t)=q,\,\,\,{\rm with}\,\,\,g(0)=h,
\end{equation}
where $q$ is a symmetric 2-tensor and $h$ is the initial metric. The associated {\em $q$-solitons} are given by
\begin{equation}\label{q soliton}
\frac{1}{2}\mathcal{L}_Xg=\lambda g+\frac{1}{2}q,
\end{equation}
where $X$ is a vector field and $\lambda$ is a constant. When $X=\nabla f$ for some smooth function $f$ on $M$, the $q$-soliton is said to be a {\em gradient $q$-soliton} and equation \eqref{q soliton} becomes
\begin{equation}\label{grad q soliton}
{\rm Hess} f=\lambda g+\frac{1}{2}q,
\end{equation}
where the Hessian, ${\rm Hess} f$, is the matrix of second order partial derivatives of the {\em potential function} $f$. The coefficient of $q$ was chosen to show explicitly that gradient solitons are self-similar solutions to the q-flow in\cite[Theorem 3.13]{Griffin}. (More generally, in \cite{Lauret1, Lauret2}, Lauret shows that solitons are self similar solutions to the flow.) It is easy to show that this definition aligns with definitions considered by several authors (see for example \cite{Kar}, \cite{Ho}).

As with others solitons (e.g. Ricci solitons, Yamabe solitons, Ricci-Bourguignon solitons, etc.), we say that the $q$-soliton is {\em expanding}, {\em steady}, and {\em shrinking} if $\lambda<0$, $\lambda=0$, and $\lambda>0$, respectively. We also say that the $q$-soliton is {\em stationary} if $X$ is a Killing vector or, in the gradient case, if it has constant potential function. 

In \cite{Griffin}, the second author proved that for any  divergence-free, trace-free tensor $q$, any compact $q$-soliton is $q$-flat (see \cite[Lemma 3.6]{Griffin}). This generalized the well-known result for Ricci solitons that any compact Ricci soliton with constant scalar curvature is Einstein (see \cite{PW}). In this paper we will expand this generalization to the complete case as follows:


\begin{theorem}\label{thm Intro 01} Let $(M,g,X)$ be an $n$-dimensional complete soliton to the $q$-flow for a divergence-free, trace-free tensor $q$. Further, suppose $(M,g,X)$ has non-positive Ricci curvature. If $M$ is parabolic and $|X|\in L^\infty(M)$ or if $|X|\in L^p(M)$ for $p>1$, then $(M,g)$ is stationary and $q$-flat.
\end{theorem}

We will see in Section \ref{Q Section} that in the non-compact case it is frequently necessary to put conditions on $X$, $|X|$, and/or on the curvature of the manifold. (e.g. In the above theorem we needed our manifold to have non-positive Ricci curvature.) These are reasonable conditions to ask for that don't reduce us to trivial cases. For example, consider a gradient Bach soliton in dimension $n=4$. As we discuss later, the Bach tensor in $n=4$ is trace-free and divergence-free. As showed in \cite{Griffin}, the manifold $\R \times E(2)$ produces a Bach-flat gradient Bach soliton with potential function $f(x) = ax+b$ when a certain metric condition is met. This metric condition forces the manifold to be Ricci flat and, consequently, have non-positive Ricci curvature. Since $X = \nabla f = 0$, $\R \times E(2)$ (trivially) satisfies the conditions of Theorem \ref{thm Intro 01} showing that there are indeed interesting examples to explore that fit the curvature condition.

As in the compact case, when we focus our attention on gradient solitons we are able to drop the requirement that our tensor, $q$, be divergence-free. That is, we get results for $q$ that are only trace-free. The following is one such result.

\begin{theorem}\label{thm intro 02}
Let $(M, g, f)$ be a complete gradient $q$-soliton with non-negative sectional curvature and the gradient of $f$ satisfies $\int_M \vert \nabla f \vert ^{\frac{n}{n-1}} < \infty.$ If $q$ is trace-free, then $M$ is $q$-flat.
\end{theorem}
In Subsection \ref{no div free} we examine other conditions on $X$ that would allow us to focus on tensors that are only trace-free and conclude by looking at a quick example of such a tensor.

\subsection{Motivation}
In the second author's 2020 paper, \cite{Griffin}, she introduces a general geometric flow for a general tensor $q(g)$, \eqref{q-flow}, and uses this structure to find results for generalized solitons. She proceeds by applying these results via corollary to the ambient obstruction tensor. This approach has three major benefits that we will continue to see in this paper. As such we will discuss these benefits briefly to motivate our work.

First, this approach allows us to tackle tensors that are difficult to work with explicitly but that are valuable to understand. The ambient obstruction tensor, in particular, has a complicated explicit representation, but it is both trace-free and divergence-free. (Because these are also properties of the Cotton and Bach tensors, we often stipulate that $q$ is trace-free and/or divergence-free.) Approaching the study of ambient obstruction solitons using a general tensor allows us to focus on these properties and avoid the need to work explicitly. As for why this is valuable, we see that the ambient obstruction tensor, Bach tensor, and Cotton tensor have unique connections to conformal geometry and physics that make them important tensors to understand. Furthermore, the investigation of solitons of these flow has already produced results that differ from the study of Ricci flow. One key example is the existence of an expanding homogeneous gradient ambient obstruction soliton for $n=4$. There are no such Ricci solitons, however, in \cite[Theorem 4.12(d)]{Griffin}, the first author shows that $\R\times S^3$ produces such a Bach soliton.

Second, this approach produces a tool that can be used to find results for a various flows. After establishing results for a general tensor $q$ in Section \ref{Q Section}, we apply these results to the ambient obstruction tensor in Section \ref{AOT Section}, then  to the Cotton tensor in Section \ref{Cotton Section}, and finally to the Bach tensor ($n\geq 5$) in Section \ref{Bach Section}. The last two applications will replicate results from \cite{CN} and \cite{CLM}, respectively, but we have included these results to demonstrate the utility of our approach and to unify results.

Most importantly, this approach allows us to see the nuanced interplay between the properties of $q$ and requirements on solitons. One example arises when we compare Lemma 3.6(c) and Proposition 3.12 in \cite{Griffin}. While both have the same conclusion, Proposition 3.12 only requires that $q$ be trace-free but is limited to gradient solitons;  Lemma 3.6 requires the $q$ is trace-free and divergence-free but applies to all solitons. This nuance would have been hidden had we used the ambient obstruction flow directly, taking away our ability to apply these results to, for example, the Bach tensor for $n\geq 5$ with reduced requirements on our manifold.

Continuing to use this approach, we will see that all of these benefits continue to exist and will, additionally, enable us to collect various results and to produce a more complete resource.

\subsection{Overview of Examples}

Following the methodology set forth in \cite{Griffin}, we apply our results for a general $q$ to a few examples of tensors that are both trace-free and divergence-free. Our direct application of the theorems is intended to demonstrate the versatility of the general results and avoids the nuances that arise in the proofs in these specific settings.

In dimension $n=3$, we examine Cotton Solitons. The Cotton tensor is given by
\[ C_{ijk} = (\nabla_iS)_{jk} - (\nabla_j S)_{ik},\]
where $S$ is the Schouten tensor. The Cotton tensor vanishes only for conformally flat manifolds and is itself the unique conformal invariant in dimension 3. The (0,2)-Cotton tensor (also referred to as the Cotton-York tensor) is given by
\begin{equation} \label{CY Eqn} C_{ij} = \frac{1}{2\sqrt{g}}C_{mni}\epsilon^{nml}g_{lj},\end{equation}
where $\epsilon^{123}=1$. It is well established that this tensor is divergence-free and trace-free. Thus, it's geometric flow and resulting solitons will provide an interesting example in dimension $n=3$. (For further background on the Cotton flow, see \cite{Kisisel, GR, York}).

Another such example is the Bach tensor in dimension $n=4$ given explicitly by the equation:
\begin{equation} \label{Bach 4 Eqn} B_{ij} = \nabla^k \nabla^l W_{ikjl} + \frac{1}{2} R_{kl} W_{i\;\; j}^{\;k\;\, l}. \end{equation}
Limiting our scope (for now) to $n=4$, we note that the Bach tensor is trace-free, divergence-free, and conformally invariant of weight $-2$. \cite{Cao} Further, since the Bach tensor can be realized as the gradient of the Weyl energy in this context, any Bach-flat metrics are critical points of the Weyl energy. It is also well established that Einstein metrics and (anti) self-dual metrics are Bach-flat. Thus, we see evidence that establishing instances of Bach-flat metrics is a useful pursuit. More than just this abstract case for Bach-flat metrics, we also see in the literature that Bach-flat manifolds are a good environment to work within. (See, for example, \cite{Kim}.)

There are two higher dimensional generalizations of the Bach tensor that maintain being both trace-free and divergence-free. However, fundamental differences in the nature of these two generalizations give rise to two distinct flows that might be called the Bach flow. As such, we present results in both contexts to show what settings need to be considered in each case.

The first generalization is the ambient obstruction tensor, given below, which provides a generalization in even dimension $n\geq 4$.
\begin{equation}\label{AOT Eqn}
\begin{gathered}  
\calO_n = \frac{1}{(-2)^{\frac{n}{2}-2} \left(\frac{n}{2}-2 \right)! } \left( \Delta^{\frac{n}{2}-1} P - \frac{1}{2(n-1)} \Delta^{\frac{n}{2}-2} \nabla^2 S \right) + T_{n-1} \\
 P = \frac{1}{n-2} \left( \ric - \frac{1}{2(n-1)} S_g \right) \end{gathered}\end{equation}
where $P$ is the Schouten tensor and $T_{n-1}$ a polynomial natural tensor of order $n-1$. This generalization naturally maintains the useful properties the Bach tensor has in dimension $n=4$. \cite{FG, GH, Lopez}

The second generalization of the Bach tensor is the Bach tensor itself considered in dimension $n\geq 5$, given explicitly by: \begin{equation}\label{Bach n Eqn} B_{ij} = \frac{1}{n-3} \nabla^k \nabla^l W_{ikjl} + \frac{1}{n-2} R^{kl} W_{ikjl}.\end{equation} 
While it is well established that the Bach tensor is trace-free for all $n\geq 4$, we note that by \cite{Cao} the divergence of the Bach tensor is given by:
\[ \div B_n \equiv \nabla_j B_{ij} = \frac{n-4}{(n-2)^2} C_{ijk} R_{jk}\]
Consequently, for $n\geq 5$ one need only consider manifolds with harmonic Weyl curvature, $\delta W=0$, to have a setting in which the Bach tensor is also divergence-free.

We will provide background on each tensor, their flows, and their solitons in Sections \ref{AOT Section} and Sections \ref{Bach Section} and show how our results can be applied in either case. In the cases of the Cotton tensor and the general Bach tensor, applying our general results to these tensors enabled us to duplicate a number of established results in upcoming papers. Because these corollaries are both stand alone results in other works and corollaries to results in Section \ref{Q Section}, we included both references in the proofs of each corollary so that interested readers can see the nuances these proofs take on in their specific setting while maintaining the completeness of this paper. 



\section{Results for a general tensor} \label{Q Section}

In the following section, we will find various conditions to force flatness. We will frequently use the following fact to do so.
\begin{fact} For a trace-free tensor $q$, a stationary $q$-soliton will always be steady and $q$-flat. \end{fact}
\begin{proof}
Suppose $q$ is a trace-free tensor and $(M, g, X)$ is a stationary soliton. By definition, $X$ is Killing. Applying this to \eqref{q soliton}, we see that $0 = \lambda g + \frac{1}{2}q$. Taking the trace of both sides, we see that $0=\lambda n$, so $\lambda =  0$ and the soliton is steady. Consequently, $q=0$ and $M$ is $q$-flat. 
\end{proof}
\noindent (Note that this also holds in the gradient case because $f$ being constant forces $\hess f = 0$. )

Because this paper exclusively considers the case where $q$ is trace-free tensor, we need only show that $X$ is Killing to force the soliton to be $q$-flat. However, in some instance we are able to show the stronger condition that $X$ is parallel. That is, $\nabla X=0$. It is quick to see the a parallel field is always Killing.

The following lemma will also prove to be very useful in our study of solitons. This extension of the Bochner formula can be found in \cite[Lemma 2.1]{PW} and was originally shown in \cite{Poor}.

\begin{lemma}\label{lemma01}
Let $q$ be a symmetric two tensor and $(M,g,X)$ a $q$-soliton \eqref{q soliton}. Then the following holds:
\begin{equation}\label{formula01}
\frac{1}{2}\Delta|X|^2-|\nabla X|^2+{\rm Ric}(X,X)={\rm div}(q)(X)-\frac{1}{2}\nabla_X{\rm tr}(q).
\end{equation}
\end{lemma}
\begin{proof}
From Bochner's formula in \cite{PW}, we known
$${\rm div}(\mathcal{L}_Xg)(X)=\frac{1}{2}\Delta|X|^2-|\nabla X|^2+{\rm Ric}(X,X)+\nabla_X{\rm div}(X).$$
On the other hand, taking the divergence and  trace in \eqref{q soliton} we have
$${\rm div}(\mathcal{L}_Xg)(X)={\rm div}(q)(X)\,\,\,{\rm and}\,\,\,{\rm div}X=\lambda n+\frac{1}{2}{\rm tr}(q).$$
Substituting this in Bochner's formula we conclude the proof.
\end{proof}

\subsection{Flatness via non-positive Ricci curvature}
Recall, a function, $f$, is {\em subharmonic} if $\Delta f \geq 0$. Further, a Riemannian manifold $M$ is {\em parabolic} if the unique subharmonic functions on $M$ which are bounded from above are the  constant functions. That is, if $u \in C^{\infty}(M)$ with $\Delta u \geq 0$ and $\sup_M u < +\infty$, then $u$ is constant. For more details see \cite{Grigoryan}.

In order to state our next results, let us recall a very useful auxiliary lemma due to Yau and corresponds to \cite[Theorem 3]{Yau:76}.

\begin{lemma}[Yau]\label{lemma Yau}
Let $u$ be a non-negative smooth subharmonic function on a complete Riemannian manifold $M^{n}$. If $u\in L^p(M)$, for some $p > 1$, then $u$ is constant.
\end{lemma}

\noindent Here we use the notation $L^p(M) = \{u : M \rightarrow \mathbb{R} \ ; \ \int_{M}|u|^pdM < +\infty\}$ for each $p \geq 1$.

With these conventions in place, we are ready to state out first result in the non-compact case.


\begin{theorem}\label{thm012} Let $(M,g,X)$ be an $n$-dimensional complete, non-compact soliton to the $q$-flow for a divergence-free, trace-free tensor $q$. Further, suppose $(M,g,X)$ has non-positive Ricci curvature. If $M$ is parabolic and $|X|\in L^\infty(M)$ or if $|X|\in L^p(M)$ for $p>1$, then $(M,g)$ is stationary and $q$-flat.
\end{theorem}

\begin{proof} Proceeding by enumerating each of the suppositions:
Suppose $M$ is parabolic and $|X|\in L^\infty(M)$. As above, we use Lemma \ref{lemma01} together our assumptions to see that
$$\frac{1}{2}\Delta|X|^2=|\nabla X|^2-{\rm Ric}(X,X)\geq0.$$
Hence, $|X|^2$ is a subharmonic function. Since $M$ is parabolic, we get $|X|^2$ is a constant on $M$. Therefore
\begin{equation*}
|\nabla X|^2-{\rm Ric}(X,X)=0,
\end{equation*}
i.e.,  $\nabla X=0$ and $X$ is parallel. Therefore $M$ is stationary and $q$-flat.

Finally, suppose $|X|\in L^p(M)$ for $p>1$.  We again use Lemma \ref{lemma01} to conclude that $|X|^2$ is a subharmonic function on $M$. Thus, by Lemma \ref{lemma Yau}, $|X|^2$ is constant on $M$. Thus $X$ is Killing. Therefore, $M$ is stationary and $q$-flat, proving the theorem.
\end{proof}

Since the result holds in the compact case (as shown in \cite[Lemma 3.6]{Griffin}) under these conditions, Theorem \ref{thm012} and \cite[Lemma 3.6]{Griffin} together show Theorem \ref{thm Intro 01}.

\begin{theorem}\label{thm011} Let $(M,g,X)$ be an $n$-dimensional complete, non-compact soliton to the $q$-flow for a divergence-free, trace-free tensor $q$. Further, suppose the Ricci curvature of $(M,g,X)$ satisfies: 
\[-(n-1)\frac{k^2}{1+r^2} \leq \ric \leq 0\]
in the sense of quadratic forms. 
If $\int_{M\diagdown B_r(x_0)}d(x,x_0)^{-2}|X|^2<\infty$, then $(M,g)$ is stationary and $q$-flat.
\end{theorem}

\begin{proof}
Suppose the above conditions hold. Since $q$ is trace-free and divergence-free, Lemma \ref{lemma01} forces
\begin{equation}\label{eqlemma01}
\frac{1}{2}\Delta|X|^2=|\nabla X|^2-{\rm Ric}(X,X).
\end{equation}
Following the proof of \cite[Theorem 0.1]{WFR}, we consider the cut-off function $\phi_r\in C_0^\infty(B(x_0,2r))$ for $r>0$, such that
\begin{eqnarray}\label{prob 0}\left\{ \begin{array}{lllll}
0\leq\phi_r\leq1 & {\rm in}\,\,B(x_0,2r)\\
\,\,\,\,\,\phi_r=1 & {\rm in}\,\, B(x_0,r)\\
\,\,\,\,\,|\nabla\phi_r|^2\leq\dfrac{C}{r^2} & {\rm in}\,\,B(x_0,2r)\\
\,\,\,\,\,\Delta\phi_r\leq\dfrac{C}{r^2} & {\rm in}\,\,B(x_0,2r),\\
\end{array}\right.
\end{eqnarray}
where $C>0$ is a constant. We know such a cut-off function exists from Corollary 2.3 of \cite{Setti}. (See also \cite[Corollary 4.3]{Pigola}.)

Fixing the cut-off function $\phi_r$ as above and integrating \eqref{eqlemma01}, we have that
$$\int|\nabla X|^2\phi_r^2-\int{\rm Ric}(X,X)\phi_r^2=\frac{1}{2}\int\Delta|X|^2\phi_r^2.$$
Now assuming (1) and integrating by parts we obtain that
$$\frac{1}{2}\int_{B_{2r}}\Delta|X|^2\phi_r^2=\frac{1}{2}\int_{B_{2r}}(\Delta\phi_r^2)|X|^2\leq\frac{1}{2}\int_{{B_{2r}}\diagdown B_r}\frac{C}{r^2}|X|^2\to0,$$
when $r\to\infty$.
Hence
$$\int|\nabla X|^2-{\rm Ric}(X,X)=0.$$
Since the Ricci curvature is non-positive we have that $\nabla X=0$. In particular $X$ is a parallel vector field and, by tracing \eqref{q soliton}, we get $\lambda = 0$. Therefore, $M$ is stationary and $q$-flat.
\end{proof}

\subsection{Flatness via convergence to zero at infinity}
Let $(M, g)$ be a complete non-compact Riemannian manifold and let $d(\,\cdot\,,x_0):M\rightarrow[0,+\infty)$ denote the Riemannian distance of $M$, measured from a fixed point $x_0\in M$. According to \cite{Alias}, we say that a continuous function $u\in C^0(M)$ {\em converges to zero at infinity}, when it satisfies the following condition
\begin{equation*}\label{distance condition}
\lim_{d(x,x_0)\rightarrow+\infty}u(x)=0.
\end{equation*}
Using this terminology, we obtain the following result for gradient solitons.

\begin{theorem}\label{thm02}
Let $(M,g,f)$ be a complete gradient soliton to the $q$-flow, such that $q$ is a divergence-free, trace-free tensor. If $|\nabla f|$ converges to zero at infinity, then $M$ must be stationary and $q$-flat.
\end{theorem}

Observe that by limiting our scope to the gradient case we avoid putting conditions on the curvature of the manifold.

To prove this result we recall of the following lemma which is a maximum principle at infinity and corresponds to item $(a)$ of \cite[Theorem $2.2$]{Alias}.
 
\begin{lemma}[Al\'ias, et.\,al.]\label{lemma:pm at infinity} 
Let $(M,g)$ be a complete non-compact Riemannian manifold and let $X\in\mathfrak{X}(M)$ be a smooth vector field on $M$. Assume that there exists a nonnegative, non-identically vanishing function $u\in C^{\infty}(M)$ which converges to zero at infinity and such that $g(\nabla u,X)\geq0$. If ${\rm div}X\geq0$ on $M$, then $g(\nabla u,X)\equiv0$ on $M$.
\end{lemma}

Now, we are in position to prove Theorem \ref{thm02}.

\begin{proof}[Proof of Theorem \ref{thm02}]
Proceeding by contradiction, suppose $M$ is not stationary. That is, suppose the potential function $f$ is not constant. In this case, we can consider the function $u:=|\nabla f|^2$, which is nonnegative, non-identically vanishing and converges to zero at infinity. Now let us consider the smooth vector field $X:=\nabla|\nabla f|^2$ on $M$.

For this vector field we have
$$g(\nabla u,X)=|\nabla|\nabla f|^2|^2\geq0.$$
On the other hand, since $q$ is trace-free and divergence-free by \cite[Corollary 3.2]{Griffin} we see that $\ric(\nabla f, \nabla f) \geq 0$  (in fact, $\ric(\nabla f, \nabla f) = 0$).  Thus, by Bochner's formula
\[\frac{1}{2}{\rm div}X=\frac{1}{2}\Delta|\nabla f|^2=|{\rm Hess}f|^2\geq0.\]
Hence, by Lemma \ref{lemma:pm at infinity} we get $u=|\nabla f|^2$ is a constant. Since it converges to zero at infinity we can conclude that $u\equiv0$ on $M$, contradicting our initial assertions about $f$. Therefore, $M$ must be stationary and $q$-flat.
\end{proof}

Proceeding to examine the non-gradient case, we get the following result which now requires a sign condition on Ricci curvature.

\begin{theorem}\label{thm03}
Let $(M,g,X)$ be a complete soliton to the $q$-flow for a divergence-free, trace-free tensor $q$ whose  Ricci curvature is non-positive. If $|X|$ converges to zero at infinity, then $M$ must be stationary and $q$-flat.
\end{theorem}

\begin{proof}
Proceeding by contradiction, suppose that $M$ is not stationary. That is, suppose $X$ is not Killing (so it cannot identically vanish). Consider the function $v:=|X|^2$ and the smooth vector field $Y:=\nabla|X|^2$. It is not difficult to see that $v$ is nonnegative, non-identically vanishing smooth function which converges to zero at infinity, and that the vector field $Y$ satisfies
$$g(\nabla v,Y)=|\nabla|X|^2|^2\geq0.$$
Since $q$ is trace-free and divergence-free we have from Lemma \ref{lemma01} that
\begin{equation*}\label{eqlemma02}
\frac{1}{2}{\rm div}Y=\frac{1}{2}\Delta|X|^2=|\nabla X|^2-{\rm Ric}(X,X)\geq0.
\end{equation*}
Hence, we are able to apply Lemma \ref{lemma:pm at infinity} to conclude that
$$0\equiv g(\nabla v,Y)=|\nabla|X|^2|^2,$$
that is, $v=|X|^2$ is a constant on $M$. Since $v$ converges to zero at infinity, we obtain $v\equiv0$ on $M$, reaching at a contradiction. Therefore, $M$ is stationary and, since $q$ is trace-free, we have from structural equation that $\lambda=0$ and  $(M,g)$ must be $q$-flat.
\end{proof}

\subsection{Flatness via polynomial volume growth}
Let $(M,g)$ be a connected, oriented, complete non-compact Riemannian manifold. Here we denote the geodesic ball centered at $p$ and with radius $r$ by $B(p,r)$. Considering a polynomial function $\eta:(0,+\infty)\rightarrow(0,+\infty)$, we say that $(M,g)$ has {\em polynomial volume growth} like $\eta(r)$ if there exists $p\in M$ such that
$${\rm vol}(B(p,r))= O(\eta(r)),$$
as $r\rightarrow+\infty$, where ${\rm vol}$ denotes the volume related to the metric $g$. It was observed in \cite[Section $2$]{Caminha} that we always have
\[\dfrac{{\rm vol}(B(p,r))}{\eta(r)}\geq\dfrac{{\rm vol}(B(q,r-d))}{\eta(r-d)}.\dfrac{\eta(r-d)}{\eta(r)}\]
for $p,q\in M$ at distance $d$ from each other. Consequently, the choice of $p$ does not matter for volume growth so we will just say that $(M,g)$ has polynomial volume growth.

Keeping this in mind, in the next result we obtain $q$-flat metrics by assuming some regularity for the gradient and Hessian of the potential function $f$, via a maximum principle related to polynomial volume growth due Al\'ias et al. \cite[Theorem 2.1.]{Caminha}.

\begin{theorem}\label{thm04}
Let $(M,g,X)$ be a complete non-compact $q$-soliton for some trace-free and divergence-free tensor $q$, whose Ricci curvature satisfies ${\rm Ric}\leq-\alpha g$, for some positive constant $\alpha$. If $(M,g)$ has polynomial volume growth and $|X|,|\nabla X|\in L^\infty(M)$, then $M$ must be stationary and $q$-flat.
\end{theorem}

\begin{proof}
Proceeding by contradiction, suppose $M$ is not stationary. Then, as in the proof of Theorem \ref{thm03}, we consider the nonnegative smooth function $u:=|X|^2$ and the smooth vector field $Y:=\nabla|X|^2$. From Lemma \ref{lemma01} we have
$$\frac{1}{2}{\rm div}Y=\frac{1}{2}\Delta|X|^2=|\nabla X|^2-{\rm Ric}(X,X)\geq\alpha u.$$
Also,
$$g(\nabla u,Y)=|\nabla|X|^2|^2\geq0.$$
On the other hand, from Kato's inequality we also get
$$|Y|=2|X||\nabla|X||\leq2|X||\nabla X|\ll+\infty,$$
since $|X|,|\nabla X|\in L^\infty(M)$. Hence, since $(M,g)$ has polynomial volume growth, we can invoke \cite[Theorem 2.1.]{Caminha} to conclude that $u$ vanishes identically on $M$, leading us to a contradiction. Therefore, $M$ must be stationary and $q$-flat.
\end{proof}

\subsection{Flatness via stochastic completeness}

Recall that a (not necessarily complete) Riemannian manifold $(M,g)$ is said to be {\em stochastically complete} if the heat kernel $p(x,y,t)$ of the Laplace-Beltrami operator, $\Delta$, satisfies
\begin{equation}\label{def.sto.com}
\int_{M}p(x,y,t)d\mu(y)=1,
\end{equation}
for some $(x,t)\in M\times(0,+\infty)$.
From a probabilistic viewpoint, stochastic completeness is the property of a stochastic process to have infinite life time. Further, if we consider a Brownian motion on a manifold, the property \eqref{def.sto.com} gives us that the full probability of the particle to be obtained in the state space is at most equal to one (see~\cite{Emery:89, Grigoryan, Grigoryan:89,Stroock:2000}).

In this way we can state a weaker version of the Theorem \ref{thm04} without requiring $|\nabla X|\in L^\infty$ and polynomial volume growth by only assuming stochastic completeness. In \cite{Pigola:03, Pigola:05}, Pigola, Rigoli and Setti show that stochastic completeness turns out to be equivalent to the validity of a weak form of the Omori-Yau maximum principle (see~\cite[Theorem 1.1]{Pigola:03} and~\cite[Theorem 3.1]{Pigola:05}), as it is expressed below.

\begin{lemma}\label{lema:PRS}
A Riemannian manifold $M$ is stochastically complete if, and only if, for every $u\in C^2(M)$ satisfying $\sup_Mu<+\infty$ there exists a sequence of points $\{x_k\}\subset M$ such that
$$\lim_{k}u(x_k)=\sup_Mu \quad \quad \mbox{and} \quad \quad \limsup_{k}\Delta u(x_k)\leq 0.$$
\end{lemma}

In this sense by applying Lemma~\ref{lema:PRS} we are able to prove the following:

\begin{theorem}\label{stochastic}
Let $(M,g,X)$ be a stochastically complete soliton to the $q$-flow such that $q$ is trace-free and divergence-free, and Ricci curvature satisfies ${\rm Ric}\leq-\beta g$, for some positive constant $\beta$. If $|X|\in L^\infty(M)$, then $(M,g)$ must be stationary and  $q$-flat.
\end{theorem}
\begin{proof}
An application of Lemma \ref{lemma01} jointly with Kato's inequality give us that
\begin{eqnarray*}
|X|\Delta|X|+|\nabla|X||^2 &=& \frac{1}{2}\Delta|X|^2=|\nabla X|^2-{\rm Ric}(X,X)\\
&\geq& |\nabla|X||^2 +\beta|X|^2,
\end{eqnarray*}
so that,
$$\Delta|X|\geq\beta|X|.$$

Finally, if $\sup_M|X|>0$ and assuming $|X|\in L^\infty(M)$ and that $M$ is stochastically complete, from Lemma~\ref{lema:PRS} there is a sequence of points $\{x_k\}\subset M$ such that
$$0\geq\limsup_k\Delta|X|(x_k)\geq\beta\sup_M|X|>0,$$
but this is a contradiction. Therefore, $M$ is stationary and $q$-flat.
\end{proof}

\subsection{Flatness without divergence-free condition}\label{no div free}
In \cite[Proposition 3.12]{Griffin}, the second author shows that in the compact case the additional assumption that $X= \nabla f$ allows us to remove the condition that $q$ be divergence-free. Following Ho's work for Bach solitons in \cite{Ho}, we get the following generalization of \cite[Corollary 3.3]{Ho} for the non-compact case for a general tensor $q$.

\begin{theorem}\label{thm05}
Let $(M, g, f)$ be a non-compact gradient $q$-soliton with non-negative sectional curvature where the gradient of the potential function, $f$, satisfies
\[ \int_M \vert \nabla f \vert ^{\frac{n}{n-1}} < \infty. \] 
If $q$ is trace-free, then $M$ is $q$-flat.
\end{theorem}
\begin{proof}
Consider such a $q$-soliton, $(M, g, f)$. Taking the trace of \eqref{grad q soliton}, we get that $\Delta f = \lambda n$. Consequently, if $\lambda\geq 0$ then $f$ is subharmonic and if $\lambda<0$ then $-f$ is subharmonic. By assumption, $M$ is a complete non-compact manifold with non-negative sectional curvature and $\int_M \vert \nabla f \vert ^{\frac{n}{n-1}} < \infty$. Applying \cite[Corollary 2']{Karp}, we see that $f$ is constant and thus $\lambda=0$. Therefore $M$ is stationary and $q$-flat.
\end{proof}

When combined with the compact case, \cite[Proposition 3.12]{Griffin}, we see that we can generalize this result to the complete case as stated in Theorem \ref{thm intro 02}.

The following is another application of Lemma \ref{lemma Yau}.

\begin{theorem}\label{thm06}
Let $(M, g, f)$ be a non-expanding complete gradient $q$-soliton with nonnegative potential function, $f$, such that $f\in L^p(M)$ for some $p>1$. If $q$ is trace-free, then $M$ is $q$-flat.
\end{theorem}
\begin{proof}
Since $q$ is trace-free we have from structural equation, \eqref{grad q soliton}, that
$$\Delta f=\lambda n\geq0.$$
So, from our assumptions on the potential function together Lemma \ref{lemma Yau} we obtain $f$ is a constant. Hence $\lambda=0$ and $M$ must be $q$-flat.
\end{proof}

\subsubsection{Flatness in compact conformal case}

Continuing the work from above, we also look to see what other conditions can be placed on $X$ to remove the need for $q$ to be divergence-free. To do so, we limit our scope to the compact case to investigate the effect of making $X$ a conformal vector field. 

In \cite[Lemma 3.6]{Griffin}, the second author proved that for any  $n$-dimensional compact $q$-soliton, if $q$ is divergence-free then $X$ is Killing. Additionally, she showed that if $q$ is trace-free the soliton must be $q$-flat. In the following theorem we prove a similar result that removes the requirement that $q$ be divergence-free by assuming $X$ is a conformal vector field.

\begin{theorem}\label{thm01 compact}
Let $(M, g, X)$ be a compact $q$-soliton to the $q$-flow such that $X$ is a conformal vector field. If $q$ is constant trace, then $X$ is Killing vector field. In addition if $q$ is trace-free then $(M,g)$ must be $q$-flat.
\end{theorem}
\begin{proof}
Since $X$ is a conformal vector field we have that
$$\mathcal{L}_Xg=\psi g,$$
for some smooth function $\psi$ on $M$. So, from equation \eqref{q soliton} we get
$$q=(\psi-2\lambda)g.$$
Observe that
$${\rm div}(q)=\nabla(\psi-2\lambda)=\nabla\psi,$$
and
$${\rm tr}(q)=(\psi-2\lambda)n.$$
Hence, 
\begin{equation}\label{eq tr}
    \nabla{\rm tr}(q)=n{\rm div}(q). 
\end{equation}

Since $q$ is constant trace we obtain that $q$ must be divergence-free. Because $M$ is compact, we can invoke \cite[Lemma 3.6]{Griffin} to conclude that $X$ is Killing. Applying the quick fact, if $q$ is trace-free then $M$ must be $q$-flat.
\end{proof}


Because this theorem only deals with the compact case, our examples (presented in the remaining sections) will prove to be rather uninteresting settings to see this theorem at work. Let us, therefore, briefly consider the case where $q= 2 \left( \frac{R}{n} g - \ric \right)$. That is, consider traceless Ricci solitons given by: 
\begin{equation}\label{T R soliton}
\frac{1}{2}\mathcal{L}_Xg=\lambda g+ \frac{R}{n} g - \ric.
\end{equation}
This tensor is trace-free (as the name suggests), but it is not necessarily divergence-free so \cite[Lemma 3.6]{Griffin} cannot be applied. Furthermore, \cite[Theorem 1.1]{CM} states that any compact gradient traceless Ricci soliton is trivial. Theorem \ref{thm01 compact}, then,  allows us to expand to the case where $X$ is conformal. (Note, this result can also be found in Dwivedi's paper, \cite{Dwivedi}, where they consider the traceless Ricci soliton as a type of Ricci-Bourguignon soliton.)

\begin{corollary} [Dwivedi] \label{TR cor01 compact}
Let $(M, g, X)$ be a compact traceless Ricci soliton such that $X$ is a conformal vector field. Because the tensor is trace-free, then $(M,g)$ must be stationary and $\frac{R}{n} g - \ric = 0$.  \end{corollary}

\begin{proof}
The proof follows directly from Theorem \ref{thm01 compact}. See also \cite[Theorem 1.3]{Dwivedi} for a more complete proof. 
\end{proof}

Note, since $\ric = \frac{R}{n}g$, $M$ is Einstein and $R$ is constant. Further, Dwivedi shows in \cite[Theorem 1.5]{Dwivedi} that, for such a soliton, if $X$ is nontrivial then $M$ is isometric to a Euclidean sphere.

\section{Applications to Ambient Obstruction Solitons}\label{AOT Section}

We begin our demonstration of the efficacy of our general results by applying them to ambient obstruction solitons to get a number of new results. 

As we mentioned in Section \ref{intro section}, the Bach tensor in dimension $n=4$ proves to be an interesting example of a divergence-free, trace-free tensor. Seeking a higher dimensional tensor that maintains the properties of the Bach tensor, we find the {\em ambient obstruction tensor}, $\mathcal{O}$ , defined for even dimension $n\geq4$ to be a natural choice. This tensor was introduced by Fefferman and Graham \cite{FG} and is defined explicitly by \eqref{AOT Eqn}.

Like the Bach tensor, the ambient obstruction tensor is trace-free, divergence-free, and conformally invariant of weight $2-n$. In fact, the ambient obstruction tensor for $n=4$ is the Bach tensor. Moreover, the ambient obstruction tensor is the gradient of the $Q$-energy, where $Q$ is Branson's $Q$-curvature. So, like Bach-flat metrics, ambient obstruction flat metrics are critical points of the $Q$-energy. (For further background about the ambient obstruction tensor, see \cite{BH1, FG, GH, Griffin, Lopez}).

In \cite{BH1, BH2, Lopez}, Bahuad, Helliwell, and Lopez introduce the {\em ambient obstruction flow} given by
\begin{equation}\label{O flow}
\frac{\partial}{\partial t}g(t)=\mathcal{O}_n+c_n(-1)^{\frac{n}{2}}(\Delta^{\frac{n}{2}-1}R)g,\,\,\,{\rm with}\,\,\,g(0)=h,
\end{equation}
 where $g(t)$ is a one parameter family of metrics on a Riemannian manifold of even dimension $n\geq 4$ and
\[c_n=\frac{1}{2^{\frac{n}{2}-2}\left(\frac{n}{2}-2\right)!(n-2)(n-1)}.\]
In \cite{BH1, BH2} Bahuaud and Helliwell show short-time existence and uniqueness on compact manifolds. Further, in \cite{Lopez}, Lopez expands this definition to a more general setting. 
Note, the addition of the scalar curvature term counteracts the conformal properties of the ambient obstruction tensor under the flow and is necessary to prove that the flow exists. 

For $n=4$ the flow \eqref{O flow} takes the form 
\begin{equation}
\frac{\partial}{\partial t}g(t)=B+\frac{1}{12}(\Delta R)g,\,\,\,{\rm with}\,\,\,g(0)=h.
\end{equation}
To distinguish this from our notion of the Bach flow that will be presented in Section \ref{Bach Section}, we will call this flow the {\em modified Bach flow}. This is the flow that Helliwell considers in his paper \cite{Helliwell}, in which he establishes an explicit representation for the Bach tensor on manifolds of the form $\R \times N^3$ where $N^3$ is a 3-manifold that is a unimodular Lie group. 

Using (\ref{q soliton}), one quickly sees that a Riemannian manifold $(M, g, X)$ is an {\em ambient obstruction soliton} if $g$ satisfies the equation
\begin{equation}\label{soliton O}
\frac{1}{2}\mathcal{L}_Xg=\lambda g+ \frac{1}{2}\left(\mathcal{O}_n+c_n(-1)^{\frac{n}{2}}(\Delta^{\frac{n}{2}-1}R)g\right)
\end{equation}
where $X$ is a vector field on $M$ and $c_n$ is as given above. Furthermore, in dimension $n=4$ the equation \eqref{soliton O} becomes in the {\em modified Bach soliton} given by
\begin{equation}\label{modified Bach soliton}
\frac{1}{2}\mathcal{L}_Xg=\lambda g+ \frac{1}{2}\left(B+\frac{1}{12}(\Delta R)g\right).
\end{equation}
As in (\ref{grad q soliton}), when $X=\nabla f$ the equations \eqref{soliton O} and \eqref{modified Bach soliton} become
\[\hess f=\lambda g+\frac{1}{2}\left(\mathcal{O}_n+c_n(-1)^{\frac{n}{2}}(\Delta^{\frac{n}{2}-1}R)g\right)\]
and
\[\hess f=\lambda g+\frac{1}{2}\left(B+\frac{1}{12}(\Delta R)g\right).\]

In order to use the fact that the ambient obstruction tensor is trace-free and divergence-free, we will limit ourselves to the setting where $M$ has constant scalar curvature. In so doing, we see results over homogeneous manifolds. In this setting the equations for the ambient obstruction flow and Bach flow ($n=4$) are, respectively, given by:
\[\frac{\partial}{\partial t}g(t)=\mathcal{O}_n \quad \text{and} \quad \frac{\partial}{\partial t}g(t)=B,\]
with $g(0)=h$. Furthermore, the equations for solitons are given by:
\[ \frac{1}{2}\mathcal{L}_Xg=\lambda g+ \frac{1}{2}\mathcal{O}_n \quad \text{and} \quad  \frac{1}{2}\mathcal{L}_Xg=\lambda g+ \frac{1}{2} B,\]
and the equations for gradient solitons by:
\[ \hess f=\lambda g+ \frac{1}{2}\mathcal{O}_n \quad \text{and} \quad  \hess f=\lambda g+ \frac{1}{2} B.\]
In each instance note that by stipulating that $M$ has constant scalar curvature, we reduce the problem to setting $q= \calO_n$. As such, we get the following corollaries of our results in Section \ref{Q Section}.


\begin{corollary}\label{cor AOT 012} Let $(M,g,X)$ be an $n$-dimensional complete, non-compact ambient obstruction soliton with non-positive Ricci curvature. If $M$ is parabolic and $|X|\in L^\infty(M)$ or if $|X|\in L^p(M)$ for $p>1$, then $(M,g)$ is stationary and $\calO$-flat.
\end{corollary}

\begin{proof}
This follows directly from Theorem \ref{thm012}.
\end{proof}

Like in Section \ref{Q Section}, when we consider the compact case as in \cite[Corollary 3.7]{Griffin}, we see that we can generalize this statement to the complete case. That is, any complete ambient obstruction soliton satisfying the appropriate curvature conditions and one of the three listed conditions is $\calO$-flat. 

\begin{corollary}\label{cor AOT 011} Let $(M,g,X)$ be an $n$-dimensional complete, non-compact ambient obstruction soliton. Further, suppose the Ricci curvature of $(M,g,X)$ satisfies: 
\[-(n-1)\frac{k^2}{1+r^2} \leq \ric \leq 0\]
in the sense of quadratic forms. 
If $\int_{M\diagdown B_r(x_0)}d(x,x_0)^{-2}|X|^2<\infty$, then $(M,g)$ is stationary and $\calO$-flat.
\end{corollary}

\begin{proof}
This follows directly from Theorem \ref{thm011}.
\end{proof}

\begin{corollary}\label{cor AOT 02}
Let $(M,g,f)$ be a complete gradient ambient obstruction soliton. If $|\nabla f|$ converges to zero at infinity, then $M$ must be stationary and $\calO$-flat.
\end{corollary}

\begin{proof}
This follows directly from Theorem \ref{thm02}.
\end{proof}

\begin{corollary}\label{cor AOT 03}
Let $(M,g,X)$ be a complete ambient obstruction soliton with non-positive Ricci curvature. If $|X|$ converges to zero at infinity, then $M$ must be stationary and $\calO$-flat.
\end{corollary}

\begin{proof}
This follows directly from Theorem \ref{thm03}.
\end{proof}

\begin{corollary}\label{cor AOT 04}
Let $(M,g,X)$ be a complete non-compact ambient obstruction soliton whose Ricci curvature satisfies ${\rm Ric}\leq-\alpha g$, for some positive constant $\alpha$. If $(M,g)$ has polynomial volume growth and $|X|,|\nabla X|\in L^\infty(M)$, then $M$ must be stationary and $\calO$-flat.
\end{corollary}

\begin{proof}
This follows directly from Theorem \ref{thm04}.
\end{proof}

\begin{corollary}\label{cor AOT stochastic}
Let $(M,g,X)$ be a stochastically complete ambient obstruction soliton whose Ricci curvature satisfies ${\rm Ric}\leq-\beta g$, for some positive constant $\beta$. If $|X|\in L^\infty(M)$, then $M$ must be stationary and  $\calO$-flat.
\end{corollary}

\begin{proof}
This follows directly from Theorem \ref{stochastic}.
\end{proof}

\begin{corollary}\label{cor AOT 05}
Let $(M, g, f)$ be a complete gradient ambient obstruction soliton with non-negative sectional curvature and the gradient of $f$ satisfies
\[ \int_M \vert \nabla f \vert ^{\frac{n}{n-1}} < \infty. \] 
Since $\calO$ is trace-free, $M$ is $\calO$-flat.
\end{corollary}

\begin{proof}
This follows directly from Theorem \ref{thm intro 02}.
\end{proof}

\begin{corollary}\label{cor AOT 06}
Let $(M, g, f)$ be a non-expanding complete gradient ambient obstruction soliton with nonnegative potential function and such that $f\in L^p(M)$ for some $p>1$. Since $\calO$ is trace-free, $M$ is $\calO$-flat.
\end{corollary}

\begin{proof}
This follows directly from Theorem \ref{thm06}.
\end{proof}

\begin{corollary}\label{cor AOT 01 compact}
Let $(M, g, X)$ be a compact ambient obstruction soliton such that $X$ is a conformal vector field. Since $\calO$ is trace-free then $X$ is Killing and $(M,g)$ must be $\calO$-flat.
\end{corollary}
\begin{proof}
This follows directly from Theorem \ref{thm01 compact}.
\end{proof}

Note, this corollary is in fact a weaker version of \cite[Corollary 3.7]{Griffin}. Corollary 3.7 states that any compact ambient obstruction soliton is $\calO$-flat, so placing any sort of requirement on our vector field is unnecessary. 


\section{Applications to Cotton Solitons}\label{Cotton Section}
We continue to demonstrate the efficacy of our program by applying our results to the the Cotton flow in dimension $n=3$. In so doing, we see that we duplicate a number of results that were established by the first author and Silva Jr. in their upcoming paper, \cite{CN}. 

For this let us consider $g(t)$ to be a one-parameter family of metrics on a Riemannian manifold $(M^3,g)$. Using the definitions set forth in \cite{Kisisel}, the  {\em Cotton flow} is given by
\begin{equation}\label{Cotton flow}
\frac{\partial}{\partial t}g(t)= \kappa C_{g(t)},
\end{equation}
where $C_{g(t)}$ is the $(0,2)$-Cotton tensor of  $(M^3,g(t))$ given by \eqref{CY Eqn}. When $\kappa=1$ its solutions are called {\em Cotton solitons} and they are given by   
\begin{equation}\label{X}
C+\mathcal{L}_Xg=\lambda g,
\end{equation}
where $\mathcal{L}$ denotes the Lie derivative and $\lambda$ is a constant. Note that these solitons inherit the sign conventions on $\lambda$ set forth in Section \ref{intro section}. If $X$ is a Killing vector field the Cotton soliton is said to be {\em trivial}. Further, if $X=\nabla f$ for some smooth function $f:M\to\mathbb{R}$ on $M^3$, we get a {\em gradient Cotton soliton} and equation \eqref{X} becomes
\begin{equation}\label{Cotton solitons}
C+2{\rm Hess}f=\lambda g.
\end{equation}

Reconciling these soliton definitions with those of \eqref{q soliton} and \eqref{grad q soliton}, we see that we will need to let $q= -2C$. Functionally, this doesn't change anything because $-2C$ is still trace-free and divergence-free.

In  \cite{GR}  García-R\'io et al. showed that any compact Cotton soliton must be trivial. They also showed the existence of (complete) non-trivial Riemannian shrinking Cotton solitons on the Heisenberg group. More recently, in \cite{CN}, the first author and Silva Jr. studied complete non-compact Cotton solitons and they obtained triviality by assuming some integral condition or $L^p$ and $L^\infty$ regularity on the vector field $X$. Consequently, the reader should note that many of the following corollaries duplicate results in \cite{CN}. As such, the proofs will contain references to both the the appropriate result(s) in \cite{CN} and the theorem in Section \ref{Q Section} that they are a corollary to. For the reader interested in understanding the nuances that arise when examining Cotton solitons specifically, we recommend looking at the proofs given in \cite{CN}.


\begin{corollary} [Cunha-Silva Jr.] \label{cor C 012} Let $(M^3,g,X)$ be an $3$-dimensional complete, non-compact Cotton soliton with non-positive Ricci curvature. If $M$ is parabolic and $|X|\in L^\infty(M)$ or if $|X|\in L^p(M)$ for $p>1$, then $(M,g)$ is stationary and $q$-flat.
\end{corollary} 

\begin{proof}
This follows directly from Theorem \ref{thm012}. See also \cite[Theorems 4]{CN} for a more complete proof.
\end{proof}

\begin{corollary} [Cunha-Silva Jr.] \label{cor C 011} Let $(M^3,g,X)$ be an $3$-dimensional complete, non-compact Cotton soliton. Further, suppose the Ricci curvature of $(M,g,X)$ satisfies: 
\[-(n-1)\frac{k^2}{1+r^2} \leq \ric \leq 0\]
in the sense of quadratic forms. 
If $\int_{M\diagdown B_r(x_0)}d(x,x_0)^{-2}|X|^2<\infty$, then $(M,g)$ is stationary and $q$-flat.
\end{corollary}

\begin{proof}
This follows directly from Theorem \ref{thm011}. See also \cite[Theorems 3]{CN} for a more complete proof.
\end{proof}

\begin{corollary} [Cunha-Silva Jr.] \label{cor C 02}
Let $(M^3,g,f)$ be a complete  gradient Cotton soliton. If $|\nabla f|$ converges to zero at infinity, then $M^3$ must be stationary and Cotton flat.
\end{corollary}

\begin{proof}
This follows directly from Theorem \ref{thm02}. See also \cite[Theorems 5]{CN} for a more complete proof.
\end{proof}

\begin{corollary} [Cunha-Silva Jr.] \label{cor C 03} 
Let $(M^3,g,X)$ be a complete Cotton soliton with non-positive Ricci curvature. If $|X|$ converges to zero at infinity, then $M^3$ must be stationary and Cotton-flat.
\end{corollary}

\begin{proof}
This follows directly from Theorem \ref{thm03}. See also \cite[Theorems 6]{CN} for a more complete proof.
\end{proof}

\begin{corollary} [Cunha-Silva Jr.] \label{cor C 04}
Let $(M^3,g,X)$ be a complete non-compact Cotton soliton whose Ricci curvature satisfies ${\rm Ric}\leq-\alpha g$, for some positive constant $\alpha$. If $(M^3,g)$ has polynomial volume growth and $|X|,|\nabla X|\in L^\infty(M)$, then $M^3$ must be stationary and Cotton flat.
\end{corollary}

\begin{proof}
This follows directly from Theorem \ref{thm04}. See also \cite[Theorem 7]{CN} for a more complete proof.
\end{proof}

\begin{corollary} [Cunha-Silva Jr.] \label{cor C stochastic}
Let $(M^3,g,X)$ be a stochastically complete Cotton soliton whose Ricci curvature satisfies ${\rm Ric}\leq-\beta g$, for some positive constant $\beta$. If $|X|\in L^\infty(M^3)$, then $(M^3,g)$ must be stationary and  Cotton-flat.
\end{corollary}

\begin{proof}
This follows directly from Theorem \ref{stochastic}. See also \cite[Theorem 8]{CN} for a more complete proof.
\end{proof}

\begin{corollary}\label{cor C 05}
Let $(M^3, g, f)$ be a complete gradient Cotton soliton with non-negative sectional curvature and the gradient of $f$ satisfies
\[ \int_M \vert \nabla f \vert ^{\frac{3}{2}} < \infty. \] 
Since the Cotton tensor is trace-free, $M^3$ is Cotton flat.
\end{corollary}

\begin{proof}
This follows directly from Theorem \ref{thm intro 02}.
\end{proof}

\begin{corollary}\label{cor C 06}
Let $(M^3, g, f)$ be a non-expanding complete gradient Cotton soliton with nonnegative potential function and such that $f\in L^p(M^3)$ for some $p>1$. Since the Cotton tensor is trace-free, $M^3$ is Cotton flat.
\end{corollary}

\begin{proof}
This follows directly from Theorem \ref{thm06}.
\end{proof}

In the compact case, Garcia-R\'io et al. showed in \cite[Theorem 1]{GR} that any compact Cotton soliton $(M^3,g)$ must be Cotton flat and $X$ is a Killing vector field. That said, a corollary of Theorem \ref{thm01 compact} is unnecessary. 


\section{Applications to Bach Solitons}\label{Bach Section}

Continuing our investigation of the Bach flow, we will use a slight change in conventions to align with the current literature on the Bach flow as it exists separately from the ambient obstruction flow. That is, we will use the conventions set forth by \cite{CLM}, \cite{Ho}, \cite{Kar}, and references therein.

We consider the {\em Bach flow} as defined in \cite{Ho}:
\begin{equation}\label{Bach flow}
\frac{\partial}{\partial t}g_{ij}=-B_{ij},
\end{equation}
where $B_{ij}$ is given by \eqref{Bach n Eqn} and $g(0)=h$. Following the conventions of \cite{Shin} we see that {\em Bach solitons} are metrics satisfying by the equation: 
\begin{equation} \label{Bach soliton}
    B + \mathcal{L}_X g = \lambda g
\end{equation} 
for a vector field $X$. (It should be noted that this definition differs slightly from that used in \cite{Ho}.) As in \eqref{grad q soliton}, when $X=\nabla f$ the equation \eqref{Bach soliton} becomes
\begin{equation} \label{grad Bach soliton}
    B+ 2\hess f = \lambda g
\end{equation} 

Note that we get equations \eqref{Bach soliton} and \eqref{grad Bach soliton} by setting $q=-2B$ in equations \eqref{q soliton} and \eqref{grad q soliton}, respectively. As with the Cotton tensor in Section \ref{Cotton Section}, this rescaling does not interfere with our applications of our results to the Bach tensor. We will, however, need to make specifications 

As mentioned in Section \ref{intro section}, the Bach tensor is a trace-free tensor in all dimensions. In dimension $n=4$ it is also divergence-free, but this is not necessarily the case in dimension $n\geq 5$. However, if $M$ is a Riemannian manifold with harmonic Weyl curvature then the Cotton tensor vanishes $C_{ijk}=0$ and we can conclude from \cite[Lemma 5.1]{Cao} that in dimension $n\geq5$ the Bach tensor is divergence-free on such an $M$. Because our results in Section \ref{Q Section} rely on one or both of these properties, it should be noted that $q=-2B$ is also divergence-free and trace-free in the appropriate setting. Using these properties we are able to apply our general results to the Bach flow.

%

To give the reader a more complete picture of Bach solitons, we also point out that, via \cite[Theorem 3.4]{Ho}, for $n=4$ any gradient Bach soliton with strictly positive or strictly negative Ricci curvature is Bach-flat. Furthermore, a number of results below duplicate results found in an upcoming paper of the first author, et al. (\cite{CLM}). As such, where applicable, our proofs contain references to both the appropriate result in \cite{CLM} and the results from Section \ref{Q Section}. To the reader interested in getting a more clear pictures of Bach solitons and understanding how the general results of this paper manifest in that context, we highly recommend \cite{CLM}, \cite{CLL} and \cite{Ho}.


\begin{corollary}[Cunha, et. al.]\label{cor B 012} Let $(M,g,X)$ be an $n$-dimensional complete, non-compact Bach soliton with non-positive Ricci curvature. Additionally, for $n \geq 5$, we stipulate that M also has harmonic Weyl curvature. If $M$ is parabolic and $|X|\in L^\infty(M)$ or if $|X|\in L^p(M)$ for $p>1$, then $(M,g)$ is stationary and $q$-flat.
\end{corollary}

\begin{proof}
This follows directly from Theorem \ref{thm012}. See also \cite[Theorems 4-7]{CLM} for a more complete proof.
\end{proof}

Again, we see that we can generalize this statement to the complete case by considering the compact case as in \cite[Lemma 3.6]{Griffin}. That is, any complete Bach soliton satisfying the appropriate curvature conditions and one of the three listed conditions is Bach-flat. 

\begin{corollary}[Cunha, et. al.]\label{cor B 011}  Let $(M,g,X)$ be an $n$-dimensional complete, non-compact Bach soliton. Additionally, for $n \geq 5$, we stipulate that M also has harmonic Weyl curvature. Further, suppose the Ricci curvature of $(M,g,X)$ satisfies: 
\[-(n-1)\frac{k^2}{1+r^2} \leq \ric \leq 0\]
in the sense of quadratic forms. 
If $\int_{M\diagdown B_r(x_0)}d(x,x_0)^{-2}|X|^2<\infty$, then $(M,g)$ is stationary and $q$-flat.
\end{corollary}

\begin{proof}
This follows directly from Theorem \ref{thm011}. See also \cite[Theorems 4-7]{CLM} for a more complete proof.
\end{proof}

\begin{corollary}[Cunha, et. al.]\label{cor B 02}
Let $(M,g,f)$ be a complete gradient Bach soliton. Additionally, for $n \geq 5$, we stipulate that M also has harmonic Weyl curvature. If $|\nabla f|$ converges to zero at infinity, then $M$ must be stationary and Bach-flat.
\end{corollary}

\begin{proof}
This follows directly from Theorem \ref{thm02}. See also \cite[Corollary 1]{CLL} for a more complete proof.
\end{proof}

\begin{corollary}[Cunha, et. al.]\label{cor B 03}
Let $(M,g,X)$ be a complete Bach soliton with non-positive Ricci curvature. Additionally, for $n \geq 5$, we stipulate that M also has harmonic Weyl curvature. If $|X|$ converges to zero at infinity, then $M$ must be stationary and Bach-flat.
\end{corollary}

\begin{proof}
This follows directly from Theorem \ref{thm03}. See also \cite[Theorems 3-4]{CLL} for a more complete proof.
\end{proof}

\begin{corollary}[Cunha, et. al.]\label{cor B 04}
Let $(M,g,X)$ be a complete non-compact Bach soliton whose Ricci curvature satisfies ${\rm Ric}\leq-\alpha g$, for some positive constant $\alpha$. Additionally, for $n \geq 5$, we stipulate that M also has harmonic Weyl curvature. If $(M,g)$ has polynomial volume growth and $|X|,|\nabla X|\in L^\infty(M)$, then $M$ must be stationary and Bach-flat.
\end{corollary}

\begin{proof}
This follows directly from Theorem \ref{thm04}. See also \cite[Theorems 6-7]{CLL} for a more complete proof.
\end{proof}

\begin{corollary}[Cunha, et. al.]\label{cor B stochastic}
Let $(M,g,X)$ be a stochastically complete Bach soliton such that Ricci curvature satisfies ${\rm Ric}\leq-\beta g$, for some positive constant $\beta$. Additionally, for $n \geq 5$, we stipulate that M also has harmonic Weyl curvature. If $|X|\in L^\infty(M)$, then $(M,g)$ must be stationary and Bach-flat.
\end{corollary}

\begin{proof}
This follows directly from Theorem \ref{stochastic}. See also \cite[Theorems 9-10]{CLL} for a more complete proof.
\end{proof}

Proceeding to examine the corollaries for results that only required that $q$ be trace-free, we remind the reader that the Bach tensor is trace-free in all dimensions, As such we are able to drop the requirement on the Weyl curvature for $n\geq 5$ and apply the following results to a more general collection of Bach solitons.

\begin{corollary}\label{cor B 05}
Let $(M, g, f)$ be a complete gradient Bach soliton with non-negative sectional curvature and the gradient of $f$ satisfies
\[ \int_M \vert \nabla f \vert ^{\frac{n}{n-1}} < \infty. \] 
Since $B$ is trace-free, $M$ is Bach-flat.
\end{corollary}

\begin{proof}
This follows directly from Theorem \ref{thm intro 02} and is essentially just a combination of \cite[Theorem 3.2, Corollary 3.3]{Ho}. 
\end{proof}

\begin{corollary}\label{cor B 06}
Let $(M, g, f)$ be a non-expanding complete gradient Bach soliton with nonnegative potential function and such that $f\in L^p(M)$ for some $p>1$. Since $B$ is trace-free, $M$ is Bach-flat.
\end{corollary}

\begin{proof}
This follows directly from Theorem \ref{thm06}.
\end{proof}

Following the work of Ho, we see that in the case where $n=4$ we in fact do not need the conformal condition on $X$ \cite[Theorem 3.2]{Ho}. As such, we proceed to show the following for $n\geq 5$.

\begin{corollary}\label{cor B 01 compact}
Let $(M, g, X)$, $n\geq 5$, be a compact Bach soliton such that $X$ is a conformal vector field. Since $B$ is trace-free, $X$ is Killing and $(M,g)$ must be Bach-flat. 
\end{corollary}
\begin{proof}
This follows directly from Theorem \ref{thm01 compact}.
\end{proof}


\bibliographystyle{amsplain}

\end{document}